\documentclass{amsart} 
\usepackage{enumerate,hyperref}
\newcommand{\set}[1]{\left\{#1\right\}}
\newcommand{\norm}[1]{\left\lVert#1\right\rVert}
\newcommand{\snorm}[1]{\big\lVert#1\big\rVert}

\newcommand{\ind}[1]{1\hspace{-.28em}\mathrm{I}_{#1}}
\newcommand{\abs}[1]{\left\vert#1\right\vert}
\newcommand{\aabs}[1]{\big\lvert#1\big\rvert}

\newcommand{\eps}{\varepsilon}

\newtheorem{thm}{Theorem}[section]
\newtheorem{lem}{Lemma}[section]

\theoremstyle{remark}
\newtheorem{zau}{Remark}[section]
\theoremstyle{definition}

\newcommand{\cI}{\mathcal{I}}
\begin{document}
\title[Stochastic differential
 equation involving Wiener process and fBm]{Stochastic differential
 equation involving Wiener process and fractional Brownian motion with Hurst index  $H> 1/2$}
\author{Yu.S. Mishura}
\address{\it Department of Probability, Statistics
and Actuarial Mathematics, Mechanics and Mathematics Faculty, Kyiv Taras Shevchenko National University,
Volodymyrska, 60, 01601 Kyiv, Ukraine}
\email{myus@univ.kiev.ua}
\author{G.M. Shevchenko}
\email{zhora@univ.kiev.ua}



\begin{abstract}
We consider a mixed stochastic differential
equation driven by possibly dependent fractional Brownian motion and Brownian motion.
Under mild regularity assumptions on the coefficients, it is proved that the equation has a
unique solution.
\end{abstract}
\maketitle
\section*{Introduction}
Fractional Brownian motion (fBm) with a Hurst parameter $H\in(0,1)$ is defined formally as a continuous centered Gaussian process $B_t^H=\{B_t^H, t\geq
0\}$ with the covariance $E
B_t^HB_s^H=1/2( s^{2H}+t^{2H}-\vert
t-s\vert^{2H}).$ For $H>1/2$ it exhibits a property of long-range dependence, which makes it a popular model for long-range
dependence in natural sciences, financial mathematics etc. For this reason, equations driven by fractional Brownian motion have been an object of intensive study during the last decade.

There are two principal ways to define an integral with respect to fractional Brownian motion.

One possibility is Skorokhod, or divergence integral introduced in the fractional Brownian setting in \cite{decre98}. However this definition is not very practical: it is based on Wick rather than usual products, and unlike Brownian case, in the fractional Brownian case this makes difference when integrating non-anticipating functions because of dependence of increments. This makes this definition worthless for most applications (most notably, those in financial mathematics). Moreover, it is impossible to solve stochastic differential equations with such integral except the cases of  additive or multiplicative noise; the latter case was considered in \cite{misquasi}.

Another approach is a pathwise integral, defined first in \cite{Zah98a} for fBm with $H>1/2$ as a Young integral. The papers \cite{kub1,NR,Ruzma} were the first to prove existence and uniqueness of stochastic differential equations involving such integrals. Later the pathwise approach was extended with the help of Lyons' rough path theory to the case of arbitrary $H$ in \cite{CouQi} where also unique solvability of
equations with $H>1/4$ was proved. Numerical methods for pathwise stochastic differential equations with fBm were considered in \cite{MSa,nourdin,davie,deya}.

In this paper we focus on the following mixed stochastic differential equation
involving Wiener process and fractional Brownian motion with Hurst
index $H \in (1/2,1)$:
 \begin{equation}\label{main} X_t =X_0 +\int_0^t a(s,X_{s})ds+
\int_0^tb(s,X_{s})dW_s+\int_0^tc(s,X_{s})dB_s^H,\
t\in[0,T],\end{equation}
where the integral w.r.t.\ Wiener process is the standard It\^o integral, and the integral w.r.t.\ fBm is pathwise generalized
Lebesgue--Stieltjes, or Young integral.
The motivation to consider such equations comes e.e.\ from financial applications, where Brownian motion as a model is inappropriate because of
the lack of memory, and fractional Brownian motion with $H>1/2$ is too smooth.

Unique solvability or \eqref{main} was proved in \cite{kub2} for time-independent coefficients and zero drift, \cite{MisP} for  $H\in (3/4,1)$ and bounded coefficients, and in \cite{GN} for any $H>1/2$, but under the assumption that $W$ and $B^H$ are independent. We generalize the latter result proving  that \eqref{main} has a unique solution for any $H\in (1/2,1) $ with $W$ and $B^H$ possibly dependent. The paper is organized as follows. In Section 1, we give necessary definitions and main hypotheses. In Section 2, we define Euler approximations of \eqref{main} and establish useful facts for them. In Section 3, we prove fundamental property of Euler approximations, and Section 4 contains the main result about existence and uniqueness of solution to \eqref{main}.

\section{Basic definitions and assumptions}

\subsection{Fractional derivatives, integrals and norms}
Let ($\Omega,\mathcal{F},\mathcal{F}_t,P$) be a complete
probability space equipped with a filtration satisfying standard
assumptions. Denote $\{W_t,
t \in [0,T]\}$  a standard $\mathcal F_t$-Wiener process, and $\{B_t^H, t\geq
0\}$ an fBm adapted to the filtration $\mathcal F_t$.

To integrate with respect to fractional Brownian motion, we  use the generalized (fractional) Lebesgue-Stieltjes integral
(see \cite{NR,Zah98a}). It is defined as follows.

Consider two  continuous functions  $f$ and $g$, defined on some interval $[a,b]\subset \mathbb{R}$.
For $\alpha\in (0,1)$ define fractional derivatives
\begin{gather*}
\big(D_{a+}^{\alpha}f\big)(x)=\frac{1}{\Gamma(1-\alpha)}\bigg(\frac{f(x)}{(x-a)^\alpha}+\alpha
\int_{a}^x\frac{f(x)-f(u)}{(x-u)^{1+\alpha}}du\bigg)1_{(a,b)}(x),\\
\big(D_{b-}^{1-\alpha}g\big)(x)=\frac{e^{-i\pi
\alpha}}{\Gamma(\alpha)}\bigg(\frac{g(x)}{(b-x)^{1-\alpha}}+(1-\alpha)
\int_{x}^b\frac{g(x)-g(u)}{(x-u)^{2-\alpha}}du\bigg)1_{(a,b)}(x).
\end{gather*}
Assume that
 $D_{a+}^{\alpha}f\in L_p[a,b], \ D_{b-}^{1-\alpha}g\in
L_q[a,b]$ for some $p\in (1,1/\alpha)$, $q = p/(p-1)$.

Under these
assumptions, the generalized (fractional) Lebesgue-Stieltjes, or Young
integral $\int_a^bf(x)dg(x)$ is defined as
\begin{equation*}\int_a^bf(x)dg(x)=e^{i\pi\alpha}\int_a^b\big(D_{a+}^{\alpha}f\big)(x)\big(D_{b-}^{1-\alpha}g\big)(x)dx.
\end{equation*}
It was shown in \cite{Samko} that for any $\alpha\in(1-H,1)$ there exists the fractional derivative
$D_{b-}^{1-\alpha}B^H\in L_{\infty}[a,b].$ Hence, for
$f$ with $D_{a+}^\alpha f\in L_1[a,b]$ we can define the integral w.r.t.
fBm according to this formula:
\begin{equation}
\label{fBm}
\int_a^bf_sdB^H_s=e^{i\pi\alpha}\int_a^b\big(D_{a+}^{\alpha}f\big)(x)\big(D_{b-}^{1-\alpha}B^H\big)(x)dx.
\end{equation}

In view of this, we will consider the following two norms for $\alpha \in(1-H,1/2)$:
\begin{gather*}
\norm{f}^2_{2,\alpha,[a,b]} = \int_a^b \left(\abs{f(s)} + \int_a^s \abs{f(s)-f(z)}(s-z)^{-1-\alpha}dz\right)^2\big(s^{-\alpha} + (t-s)^{-\alpha-1/2}\big)ds,\\
\norm{f}_{\infty,\alpha,[a,b]} = \sup_{s\in [a,b]} \left(\abs{f(s)} + \int_a^s \abs{f(s)-f(z)}(s-z)^{-1-\alpha}dz\right).
\end{gather*}
It is clear that $\norm{f}_{2,\alpha,[a,b]}\le C_{\alpha,a,b}\norm{f}_{\infty,\alpha,[a,b]}$. Throughout the paper there will be no ambiguity
about $\alpha$, so for the sake of shortness we will
denote $\norm{f}_{x,t}=\norm{f}_{x,\alpha,[0,t]}$, where $x\in\set{2,\infty}$.

\subsection{Estimates for stochastic integrals and increments}

Recall that the classical Garsia--Rodemich--Rumsey inequality \cite{garsia} states that
for a function $f\in C([0,T])$ and any $p>0$, $\theta>1/p$
\begin{equation}
\sup_{0\le v<u\le T} \frac{\abs{f(u)-f(v)}}{(u-v)^{\theta-1/p}}\le C_{\alpha,p,\theta}\left(\int_0^T\!\!\int_0^T\frac{\abs{f(x)-f(y)}^{p}}
{\abs{x-y}^{\theta p+1}}\,dx\,dy\right)^{1/p}.
\end{equation}
Setting in this inequality for $\eta\in(0,1/2)$ \ $p=2/\eta$, $\theta = (1-\eta)/2$, we get that for any $t>0$, $u,v\in [0,t]$
\begin{equation}\label{mc-bm}\abs{W_u - W_v}\le K^{W,\eta}_t\abs{u-v}^{1/2-\eta},
\end{equation}
where \begin{equation}K^{W,\eta}_t= C_{\eta}\left(\int_0^t\int_0^t\frac{\abs{W_x-W_y}^{2/{\eta}}}{\abs{x-y}^{1/{\eta}}}dx\,dy\right)^{{\eta/2}},\end{equation}
$C_\eta$ is a nonrandom constant.

Similarly, we get for any $\eta\in(0,H)$
\begin{equation}\label{mc-fbm}\abs{B^H_u - B^H_v}\le  K^{B,\eta}_t\abs{u-v}^{H-\eta},
\end{equation}
where
$$
K^{B,\eta}_t= C_{H,\eta} \Big(\int_0^t\int_0^t\frac{\vert
B_x^H-B_y^H\vert^{2/\eta}}{\vert x-y
\vert^{2H/\eta}}dx\,dy\Big)^{\eta/2}.
$$
Hence it is easy to deduce that for any $\alpha\in (1-H,1/2)$, $\eps<\alpha+H-1$
$$
\sup_{0\leq
u< v\leq t}\aabs{\big(D_{v-}^{1-\alpha}B^H\big)(u)}
\leq  C_{\alpha,H,\varepsilon}K^{B,\eps}_t.
$$
Thus, thanks to \eqref{fBm}, the stochastic integral with respect to fBm  admits  the following estimate:
\begin{equation}\label{ineq}
\Big\vert \int_u^v f(s)dB_s^H\Big\vert \le
C_{\alpha} K^{B,\eps}_t \int_u^v\Big(\abs{f(s)}
(s-u)^{-\alpha}+\int_u^s \abs{f(s)-f(z)}
(s-z)^{-\alpha-1}dz\Big)ds
\end{equation}
for any
$\alpha\in(1-H,1/2) $, $t>0$, $u\le v\le t$ and any $f$ such that the right-hand side of this inequality is finite.

\subsection{Assumptions}
In what follows we will assume the following standard hypotheses.
\begin{list}{Hypotheses}{\parsep=-1mm}
\item[(A)] \emph{Linear growth}: for any $ t\in [ 0,T]$ and any $x\in
\mathbb{R}$ \begin{equation*} | a(t,x)|  +| c(t,x)| \leq
K(1+\abs{x}).\end{equation*}
\item[(B)] \emph{Lipschitz continuity of $a,b$}: for any  $t\in[0,T]$ and $x,y\in \mathbb R$
$$|a(t,x)-a(t,y)|+|b(t,x)-b(t,y)|\le K|x-y|.$$
\item[(C)] \emph{H\"older continuity in time}: the function $c(t,x)$ is differentiable in $x$ and there exists $\beta\in(1-H,1)$
such that for any ${s,t\in[0,T]}$ and any ${x\in \mathbb R}$
$$|a(s,x)-a(t,x)|+|b(s,x)-b(t,x)|+|c(s,x)-c(t,x)|+|\partial_x{}c(s,x)-\partial_x{}c(t,x)|\le{K}|s-t|^{\beta}.$$
\item[(D)] \emph{Lipschitz continuity of $\partial_x c$}: for any  $t\in[0,T]$ and any $ x,y\in \mathbb R$  $$
|\partial_x{}c(t,x)-\partial_x{}c(t,y)|\le{K}|x-y|.
$$
\item[(E)] \emph{Boundedness of $b$ and $\partial_x c$}: for any $T\in[0,T]$ and $x\in \mathbb{R}$
$$
\abs{b(t,x)} + \abs{\partial_x c(t_x)} \leq K.
$$
\end{list}
Here $K$ is  a constants independent of $x$, $y$, $s$ and $t$.

\section{Auxiliary properties of Euler approximations}

For $n\ge 1$ consider the following partition of the fixed interval $[0,T]:$
$\{0=\nu_0 < \nu_1< \dots < \nu_n = T,\ \delta=T/n\},\
\nu_k=k\delta.$

Consider Euler approximation for equation  \eqref{main}:
\begin{equation*}
\begin{gathered}
X_{\nu_{k+1}}^{\delta} =X_{\nu_k}^{\delta} +a(\nu_k,X_{\nu_k}^{\delta})(\nu_{k+1}-\nu_{k})+
b(\nu_k,X_{\nu_k}^{\delta})(W_{\nu_{k+1}}-W_{\nu_k})+c(\nu_{k},X_{\nu_{k}}^{\delta})(B_{\nu_{k+1}}^H-B_{\nu_{k}}^H),
\end{gathered}
\end{equation*}
with $X_{\nu_0}^{\delta}=X_0$.

Set $t^\delta_u=\max \{ \nu_n:\nu_n \leq u \}$ and define continuous interpolation by
\begin{equation*} X_u^{\delta} =X_{t^\delta_u}^{\delta} +a(t^\delta_u,X_{t^\delta_u}^{\delta})(u-t^\delta_u)+
b(t^\delta_u,X_{t^\delta_u}^{\delta})(W_u-W_{t^\delta_u})+c(t^\delta_u,X_{t^\delta_u}^{\delta})(B_u^H-B_{t^\delta_u}^H),\end{equation*}
or, in the integral form,
\begin{equation}\label{euler} X_u^{\delta} =X_0 +\int_0^u a(t^\delta_s,X_{t^\delta_s}^{\delta})ds+
\int_0^ub(t^\delta_s,X_{t^\delta_s}^{\delta})dW_s+\int_0^uc(t^\delta_s,X_{t^\delta_s}^{\delta})dB_s^H.
\end{equation}

Observe that the estimates \eqref{mc-bm} and \eqref{mc-fbm} together with
our main hypotheses imply that
\begin{equation}
\label{mc-euler}
\abs{X^\delta_s - X^\delta_{t_s^\delta}}\le C K^{\eta}_s(s-t^\delta_s)^{1/2-\eta}\big(1+\aabs{X_{t_s^\delta}^\delta}\big).
\end{equation}
with $K^{\eta}_s=K^{B,\eta}_s+ K^{W,\eta}_s$.
For $N\ge 1$ define a stopping time $\tau_N = \inf \set{t: K^\eta_t \ge N}\wedge T$ and a stopped
process $X^{\delta,N}_t=X^\delta_{t\wedge \tau_N}$. 

The following lemma provides an  estimate of Euler approximations, which is essential to establishing the main result, because it is independent
of $\delta$.

\begin{lem}\label{lem-aprior}
For $\alpha \in (1-H,\frac12\wedge \beta)$, $p> 0$, $N\ge 1$
\begin{equation*}
\sup_{\delta} E\left[\snorm{X^{\delta,N}}_{\infty,T}^{p}\right] <\infty.
\end{equation*}
\end{lem}
\begin{proof}
Take any  $\eta\in(0,1/2-\alpha)$. In this proof by $C$ we will denote a generic constant,
which may depend on $p$, $\alpha$, $\eta$, $T$, $\lambda$, $X_0$ and the constants from the main hypotheses, but is independent of $N$ and
$\delta$.

Denote $\norm{f}_t = \abs{f(t)} + \int_0^t \abs{f(t)-f(s)}(t-s)^{-1-\alpha} ds$,
so that $\norm{f}_{\infty,t} = \sup_{s\in[0,t]}\norm{f}_s$.

Now fix any $t<\tau_N$.
Write
\begin{gather*}
\snorm{X^{\delta}}_t \le \abs{X_0} + I_a(t) + I_b(t) + I_c(t):=\\
\abs{X_0}+ \norm{\int_0^{\cdot} a(t_s^\delta,X_{t_s}^\delta) ds}_{t}+
\norm{\int_0^{\cdot} b(t_s^\delta,X_{t_s}^\delta) dW_s}_{t}+\norm{\int_0^{\cdot} c(t_s^\delta,X_{t_s}^\delta) dB_s^H}_{t}.
\end{gather*}
Estimate
\begin{gather*}
I_a(t) \le
\int_0^{t}\aabs{a(t_s^\delta,X_{t_s}^\delta)} ds +
\int_0^{t}{\int_{s}^{t}\aabs{a(t_v^\delta,X_{t_v}^\delta)}dv}
{\big(t-s\big)^{-1-\alpha}}ds\\
\le C\left(\int_0^t \big(1+\aabs{X_{t^\delta_s}^{\delta}}\big) ds +
\int_0^{t}{\int_{s}^{t}\big(1+\aabs{X_{t^\delta_v}^\delta}\big)dv}
{\big(t-s\big)^{-1-\alpha}}ds\right)\\
\le C \left(1+ \int_0^t \snorm{X^{\delta}}_{\infty,v} ds +
\int_0^{t} \snorm{X^{\delta}}_{\infty,v}
{(t-v)^{-\alpha}}dv\right)\\
\le C\left(1+ \int_0^{t} \snorm{X^{\delta}}_{\infty,v}
{(t-v)^{-\alpha}}dv\right).
\end{gather*}
Further, write
\begin{gather*}
I_c(t) \le I_c'+I_c'':=   \abs{\int_0^{u} c(t_s^\delta,X_{t_s}^\delta) dB_s^H} +
\int_0^{t} \abs{\int_{s}^{t} c(t_v^\delta,X_{t_v}^\delta) dB^H_v}(t-s)^{-1-\alpha} ds.
\end{gather*}
By \eqref{ineq} and \eqref{mc-euler} (recall that $t<\tau_n$)
\begin{gather*}
I_c' \le C K_{t}^{\eta} \int_0^{t} \bigg(\aabs{c(t_s^\delta,X_{t_s}^\delta)} s^{-\alpha} +
\int_0^s \aabs{c(t_s^\delta,X_{t_s}^\delta)-c(t_z^\delta,X_{t_z}^\delta)} (s-z)^{-1-\alpha} dz\bigg) ds  \\
\le C N \int_0^t \bigg(\big(1+\snorm{X^{\delta}}_{\infty,s}\big) s^{-\alpha} +
\int_0^{t_s^\delta} \big((t_s^\delta-t_z^\delta)^\beta + \aabs{X^{\delta}_{t_s^\delta}-X^{\delta}_{t_z^\delta}}\big)
(s-z)^{-1-\alpha} dz\bigg) ds \\
\le C N  \bigg(1+ \int_0^t \snorm{X^{\delta}}_{\infty,s}s^{-\alpha} ds
+ \int_0^t\int_0^{t_s^\delta} \big( (s-z)^\beta + (z-t_z^\delta)^\beta +\delta^\beta \big)(s-z)^{-1-\alpha} dz\,ds\\
+ \int_0^t\int_0^{t_s^\delta} \big( \aabs{X_s^{\delta}-X_z^{\delta}} +  \aabs{X_s^{\delta}-X_{t_s^{\delta}}^{\delta}}
 + \aabs{X_z^{\delta}-X_{t_z^\delta}^{\delta}}\big)(s-z)^{-1-\alpha} dz\bigg)\\
\le C N  \bigg(1+ \int_0^t \snorm{X^{\delta}}_{\infty,s}s^{-\alpha} ds
+ t^{\beta-\alpha + 1} + \delta^{\beta-\alpha}
+ \int_0^t \snorm{X^{\delta}}_{\infty,s} ds \\
+ CN \int_0^t \int_0^{t_s^\delta}\Big[\big(1+\aabs{X^{\delta}_{t_s^\delta}}\big)
(s-t_s^\delta)^{1/2-\eta}+\big(1+\aabs{X^{\delta}_{t_z^\delta}}\big)
(z-t_z^\delta)^{1/2-\eta}\Big](s-z)^{-1-\alpha} dz\,ds\bigg)\\
\le C N^2  \bigg(1+ \int_0^t \snorm{X^{\delta}}_{\infty,s}s^{-\alpha} ds
+ \int_0^t \snorm{X^\delta}_{\infty,s} (s-t_s^\delta)^{1/2-\eta-\alpha} ds + J_c\bigg).
\end{gather*}
Here the inequality $\int_0^t \int_0^{t_s^\delta} (z-t_z^\delta)^\delta (s-z)^{-1-\alpha}dz\,ds<C\delta^{\beta -\alpha}$ is obtained
similarly to the following estimate of $J_c$, in which we change the order of integration noting that $z<t_s^\delta\Longleftrightarrow  t_z^\delta+\delta\le s$:
\begin{gather*}
J_c =  \int_0^t\int_0^{t_s^\delta} \big(1+\aabs{X^{\delta}_{t_z^\delta}}\big)
(z-t_z^\delta)^{1/2-\eta}(s-z)^{-1-\alpha} dz\,ds\\ =\int_0^{t_t^\delta} \big(1+\aabs{X^{\delta}_{t_z^\delta}}\big)
(z-t_z^\delta)^{1/2-\eta} \int_{t_z^\delta+\delta}^t (s-z)^{-1-\alpha} ds\,dz\\
=\int_0^{t_t^\delta} \big(1+\aabs{X^{\delta}_{t_z^\delta}}\big)
(z-t_z^\delta)^{1/2-\eta} (t_z^\delta+\delta-z)^{-\alpha} ds\,dz\\
=\sum_{k=1}^{[t/\delta]}
\big(1+\aabs{X^{\delta}_{\nu_{k-1}}}\big)\int_{\nu_{k-1}}^{\nu_{k}} (z-\nu_{k-1})^{1/2-\eta} (\nu_{k}-z)^{-\alpha} dz\\
= \sum_{k=1}^{[t/\delta]}
\big(1+\aabs{X^{\delta}_{\nu_{k-1}}}\big)\delta\cdot\delta^{1/2-\eta-\alpha}\mathrm{B}(3/2-\eta,1-\alpha)
\le C \bigg(1+\int_0^t \snorm{X^{\delta}}_{\infty,s} ds\bigg).
\end{gather*}
Hence we can write
\begin{gather*}
I_c' \le C N^2  \bigg(1+ \int_0^t \snorm{X^{\delta}}_{\infty,s}s^{-\alpha} ds\bigg).
\end{gather*}
Further, denoting $c_s = c(t_s,X_{t_s^\delta}^\delta)$,
\begin{gather*}
I''_c 
\le  C N  \int_0^t \int_{s}^{t} \bigg(\abs{c_v} (v-s)^{-\alpha} +
\int_{s}^v \abs{c_v - c_z}(v-z)^{-1-\alpha} dz \bigg)dv (t-s)^{-1-\alpha} ds\\
\le C N \bigg(\int_0^{t} \int_0^v (1+\snorm{X^{\delta}}_{\infty,v} )(v-s)^{-\alpha} (t-s)^{-1-\alpha} ds\,dv \\
+ \int_0^{t} \int_0^{v} \Big((t^\delta_v - t^\delta_z)^\beta + \aabs{X^\delta_{t^\delta_v} - X^\delta_{t^\delta_z}}\Big)
(v-z)^{-1-\alpha}(t-z)^{-\alpha} dz\,dv\bigg)\\
\le C N(H_1+H_2+H_3),
\end{gather*}
where
\begin{gather*}
H_1 \le \int_0^{t}\big(1+\snorm{X^{\delta}}_{\infty,v} \big) \int_0^v (v-s)^{-\alpha}(t-s)^{-1-\alpha} ds\,dv\\
\le C\int_0^t \big(1+\snorm{X^{\delta}}_{\infty,v}\big)(t-v)^{-2\alpha}dv\le C+C\int_0^t \snorm{X^{\delta}}_{\infty,v}(t-v)^{-2\alpha}dv,\\
H_2 
\le \int_0^t \int_0^{t_v^\delta} \big((v -z)^\beta+(z - t^\delta_z)^\beta\big)(v-z)^{-1-\alpha}(t-z)^{-\alpha}
\,dz\,dv\\
\le C + \int_0^t \int_0^{t_v^\delta} (z - t^\delta_z)^\beta (v-z)^{-1-\alpha}(t-z)^{-\alpha} dz\,dv,\\
H_3 
\le
\int_0^{t} \int_0^{t_v^\delta} \big(\aabs{X^{\delta}_{v} - X^{\delta}_{z}}+\aabs{X^{\delta}_{v} - X^{\delta}_{t^\delta_v}}+\aabs{X^{\delta}_{z} - X^{\delta}_{t^\delta_z}}\big)
(v-z)^{-1-\alpha} (t-z)^{-\alpha}dz\,dv\\
\le C\int_0^t \int_0^{v} \aabs{X^{\delta}_{v} - X^{\delta}_{z}}(v-z)^{-1-\alpha} dz  (t-v)^{-\alpha}  dv\\
+ C N  \int_0^t (v-t^\delta_v)^{1/2-\eta-\alpha} \big(1+\snorm{X^{\delta}}_{\infty,v}\big)(t-v)^{-\alpha} dv \\
+ C N  \int_0^t \int_0^{t_v^\delta} (z-t^\delta_z)^{1/2-\eta} \big(1+\aabs{X^{\delta}_{t_z^\delta}}\big)(v-z)^{-\alpha-1}(t-z)^{-\alpha}dz\, dv\\
\le C N \int_0^t \big(1+\snorm{X^{\delta}}_{\infty,v}\big) (t-v)^{-\alpha}dv + CN H_3'.
\end{gather*}
Here $H_3'$ is the integral appearing in the penultimate line; we skip the estimation of the last integral in $H_2$,
because it is analogous to the estimation of $H_3'$, but somewhat simpler.

Write (abbreviating $r=1/2-\eta$ and changing the order of integration)
\begin{equation}
\label{H3}
\begin{gathered}
H_3' = \int_0^{t_t^\delta} (z-t^\delta_z)^{r} \big(1+\aabs{X^{\delta}_{t_z^\delta}}\big)(t-z)^{-\alpha}\int_{t_z^\delta+\delta}^{t} (v-z)^{-\alpha-1}dv\,dz\\
\le C \int_0^{t_t^\delta} \big(1+\aabs{X^{\delta}_{t_z^\delta}}\big)(z-t^\delta_z)^{r} (t-z)^{-\alpha}(t_z^\delta+\delta-z)^{-\alpha}dz\\
= C \sum_{k=1}^{[t/\delta]} \big(1+\aabs{X^{\delta}_{\nu_{k-1}}}\big)\int_{\nu_{k-1}}^{\nu_{k}} (z-\nu_{k-1})^{r}
(\nu_{k}-z)^{-\alpha}(t-z)^{-\alpha}dz\\
\le \sum_{k=1}^{[t/\delta]-1} \big(1+\aabs{X^{\delta}_{\nu_{k-1}}}\big) (t-\nu_k)^{-\alpha}\int_{\nu_{k-1}}^{\nu_{k}} (z-\nu_{k-1})^{r}
(\nu_{k}-z)^{-\alpha}dz\\ + \big(1+\aabs{X^{\delta}_{t_t^\delta-\delta}}\big)\delta^r \int_{t_t^\delta-\delta}^{t_t^\delta}
(t_t^\delta-z)^{-\alpha} (t-z)^{-\alpha} dz\\
\le \sum_{k=1}^{[t/\delta]-1} \big(1+\snorm{X^{\delta}}_{\infty,\nu_k}\big) (t-\nu_k)^{-\alpha}ds + \big(1+\snorm{X^{\delta}}_{\infty,t_t^\delta-\delta}\big)\delta^r (t-t_t^\delta)^{1-2\alpha} \\
\le \int_0^t\big(1+\snorm{X^{\delta}}_{\infty,s}\big)(t-s)^{-2\alpha} ds.
\end{gathered}
\end{equation}
As a result,
\begin{gather*}
I_c(t)\le CN^2  \bigg(1+ \int_0^t \snorm{X^{\delta}}_{\infty,s}\big(s^{-\alpha}+(t-s)^{-2\alpha}\big) ds\bigg).
\end{gather*}

Adding the estimate for $I_a(t)$, we get for $t\le \tau_N$
\begin{gather*}
\norm{X^\delta}_t \le CN^2\bigg(1+ \int_0^t \snorm{X^{\delta}}_{\infty,s}\big(s^{-\alpha}+(t-s)^{-2\alpha}\big) ds\bigg) + I_b(t)\\
\le CN^2\bigg(1+ \sup_{s\in[0,T]} I_b(s) + t^{-2\alpha}\int_0^t \snorm{X^{\delta}}_{\infty,s} s^{-2\alpha}(t-s)^{-2\alpha} ds\bigg).
\end{gather*}

Therefore, by the generalized Gronwall lemma \cite[Lemma 7.6]{NR}, for $t\le \tau_N$
\begin{gather*}
\norm{X^\delta}_{\infty,t} \le  CN^2 \sup_{s\in[0,T]} I_b(s) \exp\set{CtN^{2/(1-2\alpha)}}\le C_{N}  \sup_{s\in[0,T]} I_b(s).
\end{gather*}
Putting $t=T\wedge \tau_N$ and using the obvious fact that $\norm{X^{\delta,N}}_{\infty,T}= \norm{X^\delta}_{\infty,T\wedge \tau_N}$, we get
\begin{gather*}
\norm{X^{\delta,N}}_{\infty,T} \le   C_{N}  \sup_{s\in[0,T]} I_b(s).
\end{gather*}
So it remains to prove that $E[\,\sup_{s\in[0,T]} I_b(s)^p]$ is bounded uniformly in $\delta$.
Write $$E[\,\sup_{s\in[0,T]} I_b(s)^p]\le I_b' + I_b'',$$ where, denoting $b^\delta_s = b(t_s^\delta,X(t_s^\delta))$,
\begin{gather*}
I_b' = E\left[\,\sup_{t\in[0,T]}\abs{\int_0^t b^\delta_s dW_s}^p \,\right]
\le C\int_0^T E\left[\,\abs{b(t_s^\delta,X(t_s^\delta))}^p \,\right]ds \le C,\\
I_b'' = E\left[\,\sup_{t\in[0,T]}\left(\int_0^t \abs{\int_s^t b(t_z^\delta,X(t_z^\delta))dW_z} (t-s)^{-1-\alpha}ds\right)^p \,\right].
\end{gather*}
It follows from the Garsia--Rodemich--Rumsey inequality that $\abs{\int_s^t b(t_z^\delta,X(t_z^\delta))dW_z}\le \xi_\delta\abs{t-s}^{1/2-\eta}$, where
$$
\xi_\delta = C\left(\int_0^T\!\!\int_0^T \frac{\abs{\int_x^y b_v^\delta dW_v}^{2/\eta}} {\abs{x-y}^{1/\eta}}dx\,dy\right)^{\eta/2}
$$
We have
$$
E[\,\xi_\delta^p\,]\le C \int_0^T\!\!\int_0^T \frac{E\left[\,\abs{\int_x^y b_v^\delta dW_v}^{2p/\eta}\right]} {\abs{x-y}^{p/\eta}}dx\,dy\\C \int_0^T\!\!\int_0^T \frac{\left|\int_x^y E\left[\,(b_v^\delta)^2\,\right] dv\right|^{p/\eta}} {\abs{x-y}^{p/\eta}}dx\,dy\le C,
$$
whence (recalling that $1/2-\eta>\alpha$)
$$
I_b'' = E [\,\xi_\delta^p\,]\sup_{t\in[0,T]}\left(\int_0^t (t-s)^{-1/2-\eta-\alpha}ds\right)^p \le C,
$$
as required.
\end{proof}

\section{A fundamental property of the sequence of Euler approximations}

Consider a pair of  partitions defined as $\{\nu_i=i T/n, 0\le i\le n \}$ 
and $\{\theta_j= j T/(n 2^{m}), 0\le j\le n 2^m \}$. 

Denote the diameters of these partitions, correspondingly, by $\delta=T/n$ and $\mu=\delta 2^{-m}$.

Let $t^\delta_u=\max \{ \nu_n:\nu_n \leq u \}$ and $t^\mu_u=\max \{
\theta_k:\theta_k \leq u \}$.
Continuous interpolations of corresponding Euler approximations
can be written in the integral form:
\begin{equation}\label{1_3}
\begin{gathered} X_u^{\delta} =X_0^{\delta} +\int_0^u a(t^\delta_s,X_{t^\delta_s}^{\delta})ds+
\int_0^ub(t^\delta_s,X_{t^\delta_s}^{\delta})dW_s+\int_0^uc(t^\delta_s,X_{t^\delta_s}^{\delta})dB_s^H,\\
X_u^{\mu} =X_0^{\mu} +\int_0^u a(t^\mu_s,X_{t^\mu_s}^{\mu})ds+
\int_0^ub(t^\mu_s,X_{t^\mu_s}^{\mu})dW_s+\int_0^uc(t^\mu_s,X_{t^\mu_s}^{\mu})dB_s^H.\end{gathered}
\end{equation}

Define, as before, a stopping time $\tau_N =\inf\set{t:K^\eta_t \ge N}\wedge T$
and $X^{\delta,N}_t = X^{\delta}_{t\wedge \tau_N}$, $X^{\mu,N}_t = X^{\mu}_{t\wedge \tau_N}$.
For $R\ge 1$ define $ B^{R,\delta,\mu}_t = \set{\norm{X^\delta}_{\infty,t} + \norm{X^\mu}_{\infty,t} \le R}$.

\begin{thm}\label{thm-fund}
Let $\alpha\in(1-H, \kappa)$, where $\kappa=\frac12\wedge\beta $. Then for any
$0<\eta<\kappa-\alpha$  and $N,R\ge 1$ the  following estimate holds
\begin{eqnarray*}
E\left[\snorm{X^{\delta,N}-X^{\mu,N}}^2_{2,T}\ind{B_T^{R,\delta,mu}}\right] \leq M_{R,N} \delta^{2(\kappa-\alpha-\varepsilon)},
\end{eqnarray*}
where the constant $M_{R,N}$ is independent of $\delta,\mu$.

Moreover, a similar estimate is valid for $E\left[\,\sup_{t\in[0,T]} \aabs{X^{\delta,N}_t-X^{\mu,N}_t}^2\,\ind{B_T^{R,\delta,mu}}\right]$.
\end{thm}
\begin{proof}
As in the proof of Lemma \ref{lem-aprior}, by $C$ we will denote a generic constant,
which may depend on $\alpha$, $\eta$, $T$, $X_0$, the constants from the main hypotheses,  but is independent of
$\delta$ and $\mu$, $N$ and $R$. For the sake of shortness for a process $Z$ define $Z_{t,s} = Z_t-Z_s$ and denote 
$r=1/2-\eta$, $h(t,s)=(t-s)^{-1-\alpha}$, $g(t,s) = s^{-\alpha} + (t-s)^{-\alpha-1/2}$, $\ind{t}=\ind{B_t^{R,\mu,\delta}}$.

It follows from \eqref{1_3} that
\begin{equation}\label{riznytsia1}
\begin{gathered}  X_{u}^{\delta,N}-X_{u}^{\mu,N}
 =\int_0^{u(N)}\!\! a_\Delta(s)ds+
\int_0^{u(N)}\!\!b_\Delta(s) dW_s+\int_0^{u(N)}\!\!c_\Delta(s)dB_s^H\\=:\cI_a(u)+\cI_b(u)+\cI_c(u),\end{gathered}
\end{equation}
where $d_\Delta(s):=d(t^\delta_s,X_{t^{\delta,N}_s}^{\delta})-d(t^\mu_s,X_{t^{\mu,N}_s}^{\mu}),\; d\in\{a,b,c\}.$
Due to our hypotheses, on
\begin{equation}\label{ddelta}
\begin{gathered}
\abs{d_\Delta(s)} \le C\Big(|t^\delta_s-t^\mu_s|^\beta +\aabs{X_{t^{\delta}_s}^{\delta,N}-X_{t^\mu_s}^{\mu,N}}\Big)
\le C\Big((s-t_s^\delta)^\beta + \aabs{X_{t^\delta_s,t^{\mu}_s}^{\delta,N}}+\aabs{X_{t^\mu_s}^{\delta,N}-X_{t^\mu_s}^{\mu,N}}\Big)\\
\le C\Big((s-t_s^\delta)^\beta + C K_t^\eta (s-t_s^\delta)^r \big(1+\aabs{X_{t_s^\delta}^{\delta,N}}\big)+\aabs{X_{s}^{\delta,N}-X_{s}^{\mu,N}}\Big).
\end{gathered}
\end{equation}

Define
$$
\snorm{f}^2_{R,t} =  \int_0^t \left(\abs{f(s)} + \int_0^s \abs{f(s)-f(z)}(s-z)^{-1-\alpha}dz\right)^2 g(t,s)
\ind{s}ds
$$
and $\Delta_t = \norm{X^{\delta,N}-X^{\mu,N}}^2_{R,t}$.

Write
\begin{gather*}
\snorm{X^{\delta,N}-X^{\mu,N}}^2_{R,t}\le 3(\norm{\cI_a}^2_{R,t}+\norm{\cI_b}^2_{R,t}+\norm{\cI_c}^2_{R,t})
\le 6(I_a' + I_a''+I_b' + I_b''+I_c' + I_c''),
\end{gather*}
where $I_d' = \int_0^{t}|\cI_d(s)|^2 g(t,s)\ind{s} ds$, $I_d''= \int_0^t \big(\int_0^s|\cI_d(s)-\cI_d(u)| h(s,u) du\big)^2 g(t,s)\ind{s} ds$, $d\in\set{a,b,c}$.
We estimate this terms one by one. Note that some first estimates may be very rough.
The reason is that we do not need them to be finer than the (apparently worse) estimates that follow.

We will need the following trivial formula, checked directly for $t>\tau_N$ and $t\le\tau_N$:
\begin{equation}\label{trivintN}
\int_0^t\int_{s(N)}^{t(N)}\! f(t,s,v)dv\, ds = \int_0^{t(N)}\!\!\!\int_{s}^{t(N)}\!\! f(t,s,v)dv\, ds = \int_0^{t(N)}\!\!\int_{0}^{v} f(t,s,v)ds\, dv.
\end{equation}

By \eqref{ddelta} and the Cauchy--Schwartz inequality, we can write
\begin{gather*}
I_a'
\le C\int_0^t\int_0^{s(N)} \big(\delta^\beta + R N \delta^r + \aabs{X_{u}^{\delta,N}-X_{u}^{\mu,N}}\big)^2 du\, g(t,s) \ind{s}ds\\
\le C\int_0^t\int_0^s \big(\delta^{2\beta} + R N \delta^{2r} + \aabs{X_{u}^{\delta,N}-X_{u}^{\mu,N}}^2\ind{u}\big) du\, g(t,s) ds
\le C_{R,N}\bigg(\delta^{2(\kappa-\eta)} +  \int_0^t \Delta_s\,g(t,s) ds\bigg).
\end{gather*}
Similarly,
\begin{gather*}
I_a''\le  C_{R,N} \int_0^t \left(\int_0^{s} \int_{u(N)}^{s(N)}\!\! \big( \delta^{\kappa-\eta} + \aabs{X_{v}^{\delta,N}-X_{v}^{\mu,N}} \big)\ind{v}dv\, h(s,u)du\right)^2 g(t,s)ds\\
\le C_{R,N} \int_0^t\bigg(\delta^{\kappa-\eta} + \int_0^s \aabs{X_{v}^{\delta,N}-X_{v}^{\mu,N}} (s-v)^{-\alpha}\ind{v} dv \bigg)^2 g(t,s)ds \\
\le C_{R,N}\bigg(\delta^{2(\kappa-\eta)} +\int_0^t \int_0^s \aabs{X_{v}^{\delta,N}-X_{v}^{\mu,N}}^2 (s-v)^{-\alpha}\ind{v} dv \,g(t,s) ds\bigg)\\
\le C_{R,N}\bigg(\delta^{2(\kappa-\eta)} +\int_0^t\Delta_s\,g(t,s) ds\bigg).
\end{gather*}
Further, by \eqref{ineq}
\begin{gather*}
I_c'\le C N \int_0^{t}\bigg[\int_0^{s(N)}\left( |c_\Delta(u)|u^{-\alpha} du  +
\int_0^u |c_\Delta(u)-c_\Delta(z)|h(u,z) dz\right)du\bigg]^2\ind{s} g(t,s) ds\\
\le C_N \int_0^t \left[\bigg(\int_0^s |c_\Delta(u)|u^{-\alpha} \ind{u} du\bigg)^2 +
\bigg(\int_0^s \int_0^u |c_\Delta(u)-c_\Delta(z)|h(u,z) dz\ind{u} du\bigg)^2 \right]g(t,s) ds\\
=: C N(J_c'+J_c'').
\end{gather*}
Analogously to $I_a'$,
\begin{gather*}
J_c'\le C_{R,N}\bigg(\delta^{2(\kappa-\eta)} +  \int_0^t \Delta_s \,g(t,s) ds\bigg).
\end{gather*}
By   \cite[Lemma 7.1]{NR},
the hypotheses (A)--(D) 
imply that
for any $t_1,t_2,x_1,\dots,x_4$
\begin{equation}\label{cocenka}
\begin{gathered}
\abs{c(t_1,x_1)-c(t_2,x_2)-c(t_1,x_3)+c(t_2,x_4)}\le K\abs{x_1-x_2-x_3+x_4}\\
+ K\abs{x_1-x_3}\abs{t_2-t_1}^\beta+ K\abs{x_1-x_3}( \abs{x_1-x_2} + \abs{x_3-x_4}).
\end{gathered}
\end{equation}
Therefore, taking into account H\"older continuity of $c$ in the first variable and Lipschitz continuity in the
second, we get
\begin{equation}
\label{cdelta}
\begin{gathered}
|c_\Delta(u)-c_\Delta(z)|= \aabs{c(t_u^\delta,X^{\delta,N}_{t_u^\delta})-c(t_z^\delta,X_{t_z^\delta}^{\delta,N})-
c(t_u^\mu,X^{\mu,N}_{t_u^\mu})+c(t_z^\mu,X_{t_z^\mu}^{\mu,N})}\\
\le \aabs{c(u,X^{\delta,N}_u)-c(z,X^{\delta,N}_z)-c(u,X^{\mu,N}_u)+c(z,X^{\mu,N}_u)} +  C\Big((u-t_u^\delta)^\beta + (z-t_z^\delta)^\beta\\
+(u-t_u^\mu)^\beta+ (z-t_z^\mu)^\beta
+\aabs{X^{\delta,N}_{u,t_u^\delta}}+\aabs{X^{\delta,N}_{z,t_z^\delta}}
+\aabs{X^\mu_{u,t_u^\mu}}+\aabs{X^\mu_{z,t_z^\mu}}\Big)\\
\le C_{R,N}\bigg(\aabs{X^{\delta,N}_u-X^{\delta,N}_z-X^{\mu,N}_u+X^{\mu,N}_z}
+ \aabs{X^{\delta,N}_{u}-X^{\mu,N}_{u}} (u-z)^\beta \\+ \aabs{X^{\delta,N}_{u}-X^{\mu,N}_{u}}
\Big(\aabs{X^{\delta,N}_{u,z}} + \aabs{X^{\mu,N}_{u,z}}\Big)+ (u-t_u^\delta)^{\kappa-\eta} + (z-t_z^\delta)^{\kappa-\eta}\bigg).
\end{gathered}
\end{equation}
Note also that $$|c_\Delta(u)-c_\Delta(z)|=\aabs{c(t_u^\mu,X_{t_u^\mu}^\mu)-c(t_z^\mu,X_{t_z^\mu}^\mu)}\le C\big(\abs{t_u^\mu-t_z^\mu}^\beta+\aabs{X_{t_u^\mu,t_z^\mu}^\mu}\big)$$ for $z\in[t_u^\delta,t_u^\mu)$ and
$|c_\Delta(u)-c_\Delta(z)|=0$ for $z\in[t_u^\mu,u)$, hence $J_c''$ admits an estimate
\begin{gather*}
J_c''\le C_{R,N} \int_0^{t}\bigg[\int_0^s\int_0^{t_u^\delta}\bigg(\aabs{X^{\delta,N}_{u}-X_{z}^{\delta,N}-X^{\mu,N}_{u}+X_{z}^{\mu,N}}
+\aabs{X^{\delta,N}_{u}-X^{\mu,N}_{u}} (u-z)^\beta\\
+ \aabs{X^{\delta,N}_{u}-X^{\mu,N}_{u}}
\Big(\aabs{X^{\delta,N}_{u,z}} + \aabs{X^{\mu,N}_{u,z}}\Big)\\+ (u-t_u^\delta)^{\kappa-\eta} + (z-t_z^\delta)^{\kappa-\eta}\bigg)h(u,z)dz\,du\bigg]^2 g(t,s)\ind{s}ds\\
+ \int_0^{t}\bigg(\int_0^s\int_{t_u^\delta}^{t_u^\mu} \Big((t_u^\mu-t_z^\mu)^\beta  + \aabs{X^{\mu,N}_{t_u^\mu,t_z^\mu}}\Big)
h(u,z) dz\,du\bigg)^2 g(t,s) \ind{s} ds\le C_{R,N}( H_1 + H_2 + H_3 + H_4).
\end{gather*}
Here we split integrand simply by the rows to estimate them apart. Write
\begin{gather*}
H_1 \le C_{R,N}\int_0^t\bigg[ \int_0^s  \bigg(\int_0^u \aabs{X^{\delta,N}_{u}-X_{z}^{\delta,N}-X^{\mu,N}_{u}+X_{z}^{\mu,N}} h(u,z) dz\bigg)^2\ind{u} du
\\+  \bigg(\int_0^s \aabs{X^\delta_{u}-X^\mu_{u}} (u-t_u^\delta)^{\beta-\alpha}\ind{u} du\bigg)^2\bigg]g(t,s)ds \le C_{R,N}\int_0^t \Delta_s\,g(t,s) ds.
\end{gather*}
Further,
\begin{gather*}
H_2\le C_{R,N}\int_0^t\left( \int_0^s \aabs{X^{\delta,N}_{u}-X^{\mu,N}_{u}} \int_0^u\Big(\aabs{X^{\delta,N}_{u,z}} + \aabs{X^{\mu,N}_{u,z}}\Big)h(u,z)dz\ind{u}du\right)^2g(t,s) ds  \\
\le C_{R,N} \int_0^t \bigg(\int_0^s \aabs{X^{\delta,N}_{u}-X^{\mu,N}_{u}}\big(\snorm{X^{\delta,N}}_{\infty,u}+
\snorm{X^{\mu,N}}_{\infty,u}\big)\ind{u} du\bigg)^2 g(t,s) ds\\
\le C_{R,N} \int_0^t \int_0^s \aabs{X^{\delta,N}_{u}-X^{\mu,N}_{u}}^2\ind{u} du\, g(t,s) ds
\le C_{R,N} \int_0^t \Delta_s \,g(t,s) ds.
\end{gather*}
Analogously to \eqref{H3},
\begin{gather*}
H_3\le C_{R,N} \int_0^t\bigg( \int_0^s\int_{0}^{t_u^\delta}\big(
(u-t_u^\delta)^{\kappa-\eta} + (z-t_z^\delta)^{\kappa-\eta}
\big)h(s,z) dz\,du\bigg)^2 g(t,s) \ind{s} ds  \\
\le C_{R,N}\int_0^t\bigg[\bigg(\int_0^s(u-t_u^\delta)^{\kappa-\alpha - \eta}du \bigg)^2  + \int_0^t \bigg(\int_0^{t_s^\delta}
(z-t_z^\delta)^{\kappa-\eta} (t_z^\delta+\delta-z)^{-\alpha} dz \bigg)^2\bigg]g(t,s) ds\\
\le C_{R,N} \int_0^t\bigg[ \delta^{2(\kappa-\alpha-\eta)} + \bigg(\sum_{k=1}^{[s/\delta]} \int_{\nu_{k-1}}^{\nu_k} (z-\nu_{k-1})^{\kappa-\eta} (\nu_k-z)^{-\alpha} dz\bigg)^2 \bigg] g(t,s) ds\\
\le C_{R,N} \bigg( \delta^{2(\kappa-\alpha-\eta)} + \int_0^t\bigg(\sum_{k=1}^{[s/\delta]} \delta\cdot \delta^{\kappa-\alpha-\eta}\bigg)^2 g(t,s) ds\bigg]\le C_{R,N} \delta^{2(\kappa-\alpha-\eta)}.
\end{gather*}
Similarly,
\begin{gather*}
H_4
\le C\int_0^t \bigg(\int_0^s \int_{t_u^\delta}^{t_u^\mu} \Big((t_u^\mu -t_z^\mu)^\beta+
R N(t_u^\mu-t_z^\mu)^r\Big)h(u,z)dz\,du\bigg)^2 g(t,s)\ind{s}ds\\
\le
C_{R,N} \int_0^t\bigg(\int_0^s\int_{t_u^\delta}^{t_u^\mu}\big((u-z)^{\kappa-\eta} + (z-t_z)^{\kappa-\eta}\big)h(u,z)dz\,du\bigg)^2g(t,s) ds
\le C_{R,N} \mu^{2(\kappa-\alpha-\eta)}.
\end{gather*}
Now turn to $I_c''$. By \eqref{ineq} and \eqref{trivintN}, we can write
\begin{gather*}
I_c'' \le N \int_0^{t} \bigg(\int_0^{s(N)}\!\! \int_{u}^{s(N)}\!\! \bigg(\abs{c_\Delta(v)}(v-u)^{-\alpha}\\ + \int_{u}^v \abs{c_\Delta(v)-c_\Delta(z)}h(v,z)dz\bigg)dv \,h(s,u)du\bigg)^2 g(t,s)\ind{s} ds
\\\le C_N \int_0^{t} \bigg(\int_0^s\bigg( \abs{c_\Delta(v)}(s-v)^{-2\alpha} + \int_0^{v} \abs{c_\Delta(v)-c_\Delta(z)}h(v,z)
(s-z)^{-\alpha}dz\bigg)\ind{v}dv\bigg)^2 g(t,s)ds \\\le C_N (L_c'+L_c''),
\end{gather*}
where we have used that $\int_0^v (v-u)^{-\alpha}h(s,u)du\le C(s-v)^{-2\alpha}$. The term $L_c'$ is estimated similarly to $I_a''$:
\begin{gather*}
L_c'
\le C_{R,N}\int_0^t\bigg(\int_0^s \Big(\delta^{\kappa-\eta} + \aabs{X_v^{\delta,N}-X_v^{\mu,N}}\Big)(s-v)^{-2\alpha}\ind{v}dv\bigg)^2 g(t,s)ds\\
\le C_{R,N}\bigg(\delta^{2(\kappa-\eta)}+\int_0^t \int_0^s  \aabs{X_v^{\delta,N}-X_v^{\mu,N}}^2 (s-v)^{-2\alpha}\ind{v} dv\,g(t,s)ds\bigg)\\
\le C_{R,N}\bigg(\delta^{2(\kappa-\eta)}+\int_0^t \Delta_s \,g(t,s)ds\bigg).
\end{gather*}

To handle $L_c''$, we use the estimate \eqref{cdelta} and proceed exactly as with $J_c''$:
\begin{gather*}
L_c''\le C_{R,N} \int_0^{t}\bigg(\int_0^s\int_0^{t_v^\delta}\bigg(\aabs{X^{\delta,N}_{v}-X_{z}^{\delta,N}
-X^{\mu,N}_{v}+X_{z}^{\mu,N}}
+\aabs{X^{\delta,N}_{v}-X^{\mu,N}_{v}} (v-z)^\beta\\
+ \aabs{X^{\delta,N}_{v}-X^{\mu,N}_{v}}
\Big(\aabs{X^{\delta,N}_{v,z}} + \aabs{X^{\mu,N}_{v,z}}\Big)\\+ (v-t_v^\delta)^{\kappa-\eta} + (z-t_z^\delta)^{\kappa-\eta}
\bigg)h(v,z)(s-z)^{-\alpha}dz\, dv\bigg)^2\ind{s}g(t,s)ds \\
+ \int_0^{t}\bigg(\int_0^s \int_{t_v^\delta}^{t_v^\mu} \Big((t_v^\mu-t_z^\mu)^\beta  + \aabs{X^{\mu,N}_{t_v^\mu,t_z^\mu}}\Big)
h(v,z) (s-z)^{-\alpha}dz\,dv\bigg)^2\ind{s} g(t,s)ds\le  C_{R,N}\sum_{k=1}^4 G_k.
\end{gather*}
Now
\begin{gather*}
G_1\le  \int_0^{t}\bigg(\int_0^s\int_0^{v}\aabs{X^{\delta,N}_{v}-X_{z}^{\delta,N}
-X^{\mu,N}_{v}+X_{z}^{\mu,N}}h(v,z)dz (s-v)^{-\alpha} \ind{v}dv\bigg)^2g(t,s)ds\\
+\int_0^t \bigg(\int_0^s\aabs{X^{\delta,N}_{v}-X^{\mu,N}_{v}} \int_0^v (v-z)^{\beta-\alpha-1}(s-z)^{-\alpha }dz\, dv\ind{v} \bigg)^2 g(t,s)ds \\
\le \int_0^{t}\int_0^s\bigg(\int_0^{v}\aabs{X^{\delta,N}_{v}-X_{z}^{\delta,N}-X^{\mu,N}_{v}+X_{z}^{\mu,N}}h(v,z)dz\bigg)^2 (s-v)^{-\alpha}\ind{v} dv\,g(t,s)ds\\ +
\int_0^t \bigg(\int_0^s\aabs{X^{\delta,N}_{v}-X^{\mu,N}_{v}} (s-v)^{\beta -2\alpha}\ind{v} dv\bigg)^2 g(t,s)ds
\le C\int_0^t\Delta_s\, g(t,s)ds.
\end{gather*}
Further,
\begin{gather*}
G_2\le C  \int_0^t\bigg(\int_0^s\aabs{X^{\delta,N}_{v}-X^{\mu,N}_{v}} \big(\snorm{X^{\delta,N}}_{\infty,v}+\snorm{X^{\mu,N}}_{\infty,v}\big)(s-v)^{-\alpha}\ind{v} dv\bigg)^2 g(t,s)ds\\
\le C_R  \int_0^t \int_0^s \aabs{X^{\delta,N}_{v}-X^{\mu,N}_{v}}^2 (s-v)^{-\alpha}\ind{v} dv\,g(t,s)ds \le C_R \int_0^t
\Delta_s\, g(t,s) ds,\\
G_3\le C  \int_0^t\bigg(\int_0^s (v-t_v^\delta)^{\kappa-\alpha-\eta}(s-v)^{-\alpha} dv\\ + \int_0^{t_s^\delta}  (z-t_z^\delta)^{\kappa-\eta} (s-z)^{-\alpha} (t_z^\delta+\delta-z)^{-\alpha}  dz\bigg)^2g(t,s)\ind{s}ds
\le C  \delta^{2(\kappa-\eta-\alpha)}.
\end{gather*}
Here the last integral is estimated analogously to the term $H_3'$ (equation \eqref{H3}) in the proof of Lemma~\ref{lem-aprior}:
\begin{gather*}
\int_0^{t_s^\delta}  (z-t_z^\delta)^{\kappa-\eta} (s-z)^{-\alpha} (t_z^\delta+\delta-z)^{-\alpha}  dz
\\\le  \sum_{k=1}^{[s/\delta]-1}(s-\nu_k)^{-\alpha}\int_{\nu_{k-1}}^{\nu_k} (z-\nu_{k-1})^{\kappa-\eta} (\nu_k-z)^{-\alpha} dz
+ \delta^{\kappa-\eta}\int_{t_s^\delta-\delta}^{t_s^\delta} (s-z)^{-\alpha} (t_s^\delta-z)^{-\alpha}  dz\\
\le C\bigg(\sum_{k=1}^{[s/\delta]-1}(s-\nu_k)^{-\alpha}\delta\cdot\delta^{\kappa-\alpha-\eta} + \delta^{\kappa-\alpha}\delta^{1-2\alpha}\bigg)
\le C\delta^{k-\alpha-\eta}\bigg(\int_0^s (s-z)^{-\alpha} dz + 1\bigg)\le  C\delta^{k-\alpha-\eta}.
\end{gather*}
Similarly,
\begin{gather*}
G_4\le  C\int_0^t\bigg(\int_0^s\int_0^{t_v^\mu} \Big((t_v^\mu-t_z^\mu)^\beta + R N(t_v^\mu-t^\mu_z)^r \Big)h(v,z)(s-z)^{-\alpha}dz\,dv\bigg)^2g(t,s)ds\\
\le C_{R,N}\int_0^{t}\bigg(\int_0^s \int_0^{t_v^\mu} \Big((v-z)^{\kappa-\eta}+(z-t_z^\mu)^{\kappa-\eta} \Big)h(v,z)dz\,(s-z)^{-\alpha} dz\,dv\bigg)g(t,s)ds \\\le C_{R,N} \mu^{2(\kappa-\alpha-\eta)}.
\end{gather*}
Summing up, we have an estimate
\begin{gather*}
\norm{\cI_a}^2_{R,t}+\norm{\cI_c}^2_{R,t} \le C_{R,N} \bigg(\delta^{2(\kappa-\alpha-\eta)} + \int_0^t \Delta_s\, g(t,s)ds\bigg).
\end{gather*}
Hence,
\begin{equation}
\label{Ia+Ic}
\begin{gathered}
E\left[\big(\norm{\cI_a}_{R,t}+\norm{\cI_c}_{R,t}\big)^2 \ind{B_t^{R}}\right]%
\le C_{R} \bigg(\delta^{2(\kappa-\alpha-2\eta)} + \int_0^t E\big[\Delta_s\big] g(t,s)ds\bigg).
\end{gathered}
\end{equation}
Now turn to $I_b'$ and $I_b''$. Using \eqref{ddelta} and noting that $B_s^{R}\subset B_u^{R}\in \mathcal F_u$ for
$u<s$, we have
\begin{gather*}
E\big[I_b' \big] = \int_0^{t} E\left[\bigg(\int_0^{s(N)} b_\Delta(u) dW(u)\bigg)^2\ind{s} \right] g(t,s)ds\\
\le  \int_0^t \int_0^s E\left[b_\Delta(u)^2\ind{u}\right]du\, g(t,s)ds\\
\le C\int_0^t\int_0^s E\left[\big(\delta^\beta + R N \delta^r + \aabs{X_u^{\delta,N}-X_u^{\mu,N}}\big)^2\ind{u}\right]du\,g(t,s)ds\\
\le C_{R,N} \bigg(\delta^{2(\kappa-\eta)} + \int_0^t E\big[\Delta_s\big]g(t,s)ds\bigg).
\end{gather*}
Further,
\begin{gather*}
E[I_b''] = \int_0^t E\left[\bigg(\int_0^{s(N)}\abs{ \int_u^{s(N)}\!\! b_\Delta(v) dW_v} (s-u)^{-1-\alpha}ds\bigg)^2
\ind{s}\right]g(t,s)ds\\
\le
\int_0^t \int_0^{s}E\left[\bigg(\int_u^{s(N)\vee u}\!\! b_\Delta(v) dW_v\bigg)^2\ind{B_s^R}\right] (s-u)^{-3/2-\alpha} du
\int_0^s (s-u)^{-1/2-\alpha} du \,g(t,s)ds
\\
\le C \int_0^t \int_0^s \int_u^s E\left[\big(\delta^\beta + R N \delta^r + \aabs{X_v^{\delta,N}-X_v^{\mu,N}}\big)^2\ind{v}\right]dv (s-u)^{-3/2-\alpha}du\,g(t,s)ds\\
\le C_{R,N} \bigg(\delta^{2(\kappa-\eta)} +
\int_0^t\int_0^s  E\big[\aabs{X_v^{\delta,N}-X_v^{\mu,N}}^2 \ind{v}\big](s-v)^{-1/2-\alpha}dv\,g(t,s)ds\bigg)\\
\le C_{R,N} \bigg(\delta^{2(\kappa-\eta)} +
\int_0^t E\big[\Delta_s\big]g(t,s)ds\bigg).
\end{gather*}
Combining this with \eqref{Ia+Ic}, we get
\begin{gather*}
E\big[\Delta_t\big]\le C_{R,N} \bigg(\delta^{2(\kappa-\alpha-\eta)} + \int_0^t E\big[\Delta_s \big] g(t,s) ds \bigg)\\
\le C_{R,N} \bigg(\delta^{2(\kappa-\alpha-\eta)} + t^{1/2+\alpha}\int_0^t E\big[\Delta_s \big] (t-s)^{-1/2-\alpha}s^{-1/2-\alpha} ds \bigg),
\end{gather*}
whence by the generalized Gronwall lemma
\begin{gather*}
E[\Delta_{t}]\le C_{R,N} \delta^{2(\kappa-\alpha-\eta)}.
\end{gather*}
Obviously, $\norm{f}_{2,T}\ind{T}\le \norm{f}_{R,T}$, so the first statement of the theorem is proved.

To prove the second, for we derive for $\omega\in B_s^R$ absolutely similarly to the previous estimates
(it is easy to check that the terms $\delta^{a}$ with $a>0$ enter with bounded coefficients \emph{before}
the integration with respect to $s$) that
\begin{gather*}
\cI_a(s)^2+\cI_c(s)^2\le C_{R,N}\big(\delta^{2(\kappa-\alpha-\eta)}+ \Delta_s\big).
\end{gather*}
The remaining term $\cI_b(s)$ estimated similarly with the help of Burkholder inequality:
$$
E\left[\sup_{u\le s}\cI_b(u)^2\ind{s}\right] \le C_{R,N}\big(\delta^{2(\kappa-\eta)}+ E\big[\Delta_s\big]\big),
$$
whence
\begin{gather*}
E\left[\sup_{s\le T}\Big(\cI_a(s)^2+\cI_b(s)^2+\cI_c(s)^2\Big)\ind{s}\right] \le C_{R,N}\big(\delta^{2(\kappa-\alpha-\eta)}+ E\big[\Delta_T\big]\le C_{R,N}\big(\delta^{2(\kappa-\alpha-\eta)},
\end{gather*}
but $(X^{\delta,N}_s -X^{\mu,N}_s)^2\le  3\big(\cI_a(s)^2+\cI_b(s)^2+\cI_c(s)^2\big)$, so the second statement follows.
\end{proof}

\section{Existence and uniqueness}
\begin{thm} Equation \eqref{main} has a solution such that for any $\alpha\in(1-H,\kappa)$
\begin{equation}\label{solution-cond}
\norm{X}^2_{\infty,\alpha,[0,T]}]<\infty\quad\text{a.s.}
\end{equation}
This solution is unique in the class of processes satisfying \eqref{solution-cond} for some $\alpha>1-H$.
\end{thm}
\begin{zau}
It is possible to generalize this result almost verbatim to a multidimensional case. Moreover, instead of fractional Brownian motion
one can take any process, which is almost surely H\"older continuous with H\"older exponent $\gamma>1/2$.
\end{zau}

\begin{proof}
As in the previous proof, we denote $Z_{t,s} = Z_t-Z_s$, $r=1/2-\eta$, $h(t,s)=(t-s)^{-1-\alpha}$, $g(t,s) = s^{-\alpha} + (t-s)^{-1/2-\alpha}$.
$C$ is a generic constant, which may depend on fixed parameters, but is independent of variables.

\emph{Existence}. Let $\delta_k = T/2^k$. To keep notations simple, denote $X^{k}=X^{\delta_k}$, $X^{k,N}=X^{\delta_k,N}$, $t_s^k = t_s^{\delta_k}$,
$\norm{\cdot}=\norm{\cdot}_{2,T}$, $\norm{\cdot}_\infty=\norm{\cdot}_{\infty}$.

It is easy to see that for each $N\ge 1$ the sequence $\set{X^{k,N}_t,k\ge 1}$ is fundamental in the norm
$\big(E[\norm{\cdot}^2]\big)^{1/2}$. Indeed,
define $A^{R,N,k,l} = \set{\norm{X^{k,N}}+\norm{X^{l,N}}\le R}$, take $p>1$, denote $q=p/(p-1)$ and
write
\begin{gather*}
E\left[\snorm{X^{k,N}-X^{l,N}}^2 \right]\\\le E\left[\snorm{X^{k,N}-X^{l,N}}^2 \ind{A^{R,N,k,l}}\right] +
E\left[\big(\snorm{X^{k,N}}_\infty+\snorm{X^{l,N}}_\infty\big)^2\ind{\Omega\setminus A^{R,N,k,l}}\right]\\
\le E\left[\snorm{X^{k,N}-X^{l,N}}^2 \ind{A^{R,N,k,l}}\right] +
\left(E\left[\big(\snorm{X^{k,N}}_{\infty}+\snorm{X^{l,N}}_{\infty}\big)^{2p}\right]\right)^{1/p}P\big({\Omega\setminus A^{R,N,k,l}}\big)^{1/q}.
\end{gather*}
By Theorem~\ref{thm-fund}, the first term vanishes as $k,l\to\infty$, so we can write
\begin{gather*}
\varlimsup_{k,l\to\infty} E\left[\snorm{X^{k,N}-X^{l,N}}^2 \right]\le \sup_{k,l}
\left(E\left[\big(\snorm{X^{k,N}}_{\infty}+\snorm{X^{l,N}}_{\infty}\big)^{2p}\right]\right)^{1/p}P({\Omega\setminus A^{R,N,k,l}})^{1/q}.
\end{gather*}
Lemma~\ref{lem-aprior} implies that $\sup_{k,l}
E\left[\big(\norm{X^{k,N}}_{\infty}+\norm{X^{l,N}}_{\infty}\big)^{2p}\right]<\infty$ and also, through
Markov inequality, that $\sup_{k,l} P({\Omega\setminus A^{R,N,k,l}})\to 0$, $R\to\infty$.
Therefore, letting $R\to\infty$, we get
\begin{gather*}
\varlimsup_{k,l\to\infty} E\left[\snorm{X^{k,N}-X^{l,N}}^2 \right] =0,
\end{gather*}
as claimed. Similarly, from the second statement of Theorem~\ref{thm-fund} we obtain
\begin{gather*}
\varlimsup_{k,l\to\infty} E\Big[\sup_{t\in[0,T]}\abs{X^{k,N}-X^{l,N}}^2 \Big] =0.
\end{gather*}

Thus, for each $N\ge 1$ there exists a process $X^N$ such that $E\left[\snorm{X^{k,N}-X^{N}}^2\right]\to 0$
and $E\big[\sup_{t\in[0,T]}\abs{X^{k,N}-X^{N}}^2 \big]\to 0$, $k\to \infty$; the limits agree because
each of this two facts implies the convergence in $L^2([0,T]\times\Omega)$. Using
a usual argument we can show that there exists a subsequence $\set{k_j,j\ge 1}$ such that for any $N\ge 1$ \
$\snorm{X^{{k_j},N}-X^N}+\sup_{t\in[0,T]}\abs{X^{k_j,N}-X^{N}}\to 0$ a.s., $j\to\infty$.
Without loss of generality we can assume that the sequence itself converges a.s. to zero:
\begin{equation}\label{as-conv}
\snorm{X^{k,N}-X^{N}}+\sup_{t\in[0,T]}\abs{X^{k,N}-X^{N}}\to 0 \quad\text{a.s. as }  k\to \infty.
\end{equation}

Note that for each $N\ge 1$, $p>0$ \ $E[\norm{X^N}^p_{\infty}]<\infty$. Indeed,
we already have a uniform convergence, so it is enough to show boundedness of the integral in the definition of $\norm{\cdot}_{\infty}$.
By the Fatou lemma,
\begin{gather*}
\int_0^t \aabs{X^N_{t,z}}(t-z)^{-1-\alpha} dz \le \varliminf_{k\to\infty} \int_0^t \aabs{X^{k,N}_{t,z}}(t-z)^{-1-\alpha} dz\le
\varliminf_{k\to\infty} \norm{X^{k,N}}_{\infty},
\end{gather*}
therefore
\begin{gather*}
\sup_{s\in[0,T]}\int_0^t \aabs{X^N_{t,z}}(t-z)^{-1-\alpha} dz\le \varliminf_{k\to\infty} \norm{X^{k,N}}_{\infty},
\end{gather*}
and applying the Fatou lemma again, we get
\begin{gather*}
E\left[\,\left(\sup_{s\in[0,T]}\int_0^t \aabs{X^N_{t,z}}(t-z)^{-1-\alpha} dz\right)^p\,\right]\le \varliminf_{k\to\infty} E\big[\norm{X^{k,N}}_{\infty}^p\big]<\infty.
\end{gather*}

Further, for $N'\le N''$ and $t\le \tau_{N'}$ \ $X^{N'}_t = X_t^{N''}$ a.s. Indeed, for any $t$ \ $X^{k,N'}_t \to X^{N''}_t$ and
$X^{k,N''}_t \to X^{N''}_t$ a.s. as $k\to \infty$; but for  $t\le \tau_{N'}$ \ $X^{k_j,N''}_t=X^{k_j,N''}_t$, so
$X^{N'}_t = X_t^{N''}$ a.s.

Hence there exists a process $X$ such that for any $N$ and any $t$ \ $X^N_t = X_{t\wedge \tau_N}$. We are going
to prove that $X$ solves \eqref{main}. This will be done by showing that each of the integrals in \eqref{euler}
converges to a corresponding integral in \eqref{main}. To this end, consider the differences between these
integrals:
\begin{gather*}
\cI^N_a(t) = \int_0^{t(N)}\! a_\Delta(s) ds,\ \cI^N_b(t) = \int_0^{t(N)}\! b_\Delta(s) dW_s,\ \cI^N_c(u) = \int_0^{t(N)}\! c_\Delta(s) dB^H_s,
\end{gather*}
here $d_\Delta(s) = d(s,X_s)-d(t_s^{k},X^{k}_{t_s^k})$, $d\in\set{a,b,c}$. Similarly to \eqref{ddelta},
\begin{equation}\label{mdelta}
\begin{gathered}
\abs{d_\Delta} \le C\Big((s-t_s^k)^\beta + C K_t^\eta (s-t_k^\delta)^r \big(1+\aabs{X_{t_s^k}^{k,N}}\big)+
\aabs{X_{s}^{k,N}-X_{s}}\Big)
\end{gathered}
\end{equation}
Estimate
\begin{gather*}
\abs{\cI^N_a(t)} \le  \int_0^{t}\! \abs{a_\Delta(s(N))} ds
\le C\Big(\delta^\beta_k +  N \delta_k^r \big(1+\snorm{X^{k,N}}\big)+
\snorm{X^{k,N}-X^{N}}\Big).
\end{gather*}
Thanks to \eqref{as-conv}, the norm $\snorm{X^{k,N}}$ is a.s. bounded in $k$. Consequently, $\cI^N_a(t)\to 0$ a.s., $k\to\infty$.

Further,
\begin{gather*}
E\left[ \big(\cI^N_b(t)\big)^2\right]\le \int_0^t E\left[b_\Delta(s)^2\right]ds \\
\le C\Big(\delta_k^\beta + \delta_k^r E\left[\big(1+\snorm{X^{k,N}}\big)^2\right]+E\left[\snorm{X^{k,N}-X^N}^2\right]\Big)\to 0, k\to\infty.
\end{gather*}
We can extract a subsequence $\set{k_j,j\ge 1}$ such that for each $N\ge 1$ \ $\cI^N_a(t)\to 0$ a.s., $j\to\infty$. Again, we will
assume without loss of generality that sequence itself vanishes.

Finally,
\begin{gather*}
\abs{\cI^N_c(t)}\le
C N \int_0^{t} \bigg(|c_\Delta(s)|s^{-\alpha} ds  +
\int_0^s |c_\Delta(s)-c_\Delta(z)|h(s,z) dz\bigg)ds =: CN\big(I_c' + I_c''\big).
\end{gather*}
Similarly to $\cI^N_a(t)$, $I_c'\to 0$, $k\to\infty$.
Using an estimate similar to \eqref{cdelta}, write
\begin{gather*}
\int_0^{t}\int_0^{t_s^\delta}\bigg(\aabs{X^N_{s}-X^N_{z}-X^{k,N}_{s}+X^{k,N}_{z}}
+\aabs{X^N_{s}-X^{k,N}_{s}} (s-z)^\beta\\
+ \aabs{X^N_{s}-X^{k,N}_{s}}
\Big(\aabs{X^N_{s,z}} + \aabs{X^{k,N}_{s,z}}\Big)+ (s-t_s^k)^\beta + (z-t_k^\delta)^\beta
+ \aabs{X^{k,N}_{s,t_s^k}}+ \aabs{X^{k,N}_{z,t_z^k}}\bigg)h(s,z)dz\,ds\\
+ \int_0^{t}\int_{t_s^\delta}^{s} \Big((s-z)^\beta  + \aabs{X^k_{s,z}}\Big)
h(s,z) dz\, ds\\
\le C N \bigg(\snorm{X^{k,N} -X^N}\big(1+ \snorm{X^N}_{\infty} + \snorm{X^{k,N}}_{\infty}\big) +\delta_k^\beta\\
{}+ \big(1+\snorm{X^{k,N}}\big)\int_0^t\!\int_0^{t_s^\delta} \big((s-t_s^\delta)^r + (z-t_z^\delta)^r\big)h(s,z)dz\,ds+ \delta_k^{\beta-\alpha} + \delta_k^{r-\alpha}\big(1+\snorm{X^{k,N}}\big)\bigg)\\
\le CN\Big(\snorm{X^{k,N} -X^N}\big(1+ \snorm{X^N}_{\infty} + \snorm{X^{k,N}}_{\infty}\big) +\delta^{\beta-\alpha}_k+\delta_k^{r-\alpha}\big(1+\snorm{X^{k,N}}_{\infty}\big)\Big)
\to 0,\ k\to\infty.
\end{gather*}
The last holds since $\snorm{X^N}_{\infty}$ is almost surely finite and $\snorm{X^{k,N}}_{\infty}$ is bounded in probability uniformly in $k$.

As a result, $\abs{\cI^N_a(s)} + \abs{\cI^N_a(s)}+\abs{\cI^N_a(s)}\to 0$, $k\to \infty$. But since $X^N_{t} - X_{t}^{k,N} = \cI_a^N(t) + \cI_b^N(t)+\cI_c^N(t)$ and $X^{k}$ solves \eqref{euler} with $\delta=\delta_k$,
we get that $X$ solves \eqref{main} up to each of the moments $\tau_N$. But (since $K_T^\eta$ is finite a.s.) $\tau_N\to T$ a.s., $N\to \infty$,
which implies that $X$ is a solution to \eqref{main}.

\emph{Uniqueness}. Let $X^1$, $X^2$ be two solutions of \eqref{main}. Define a stopping time $$\sigma_N = \inf\set{t:K_t^\eta + \norm{X^1}_{\alpha,t}+\norm{X^2}_{\alpha,t}\ge N}$$ and denote $X^{1,N}_t = X^1_{t\wedge \sigma_N}$, $X^{2,N}_t = X^2_{t\wedge \sigma_N}$,
$\Delta_s = \norm{X^{1,N}-X^{2,N}}^2_{2,s}$.

We have
\begin{gather*}
X_{u}^{1,N}-X_{u}^{2,N}
 =\int_0^{u(N)} \!\! a_\Delta (s) ds+
\int_0^{u(N)}\!\!b_\Delta (s) dW_s+\int_0^{u(N)}\!\!c_\Delta (s) dB_s^H\\=:\cI_a(u)+\cI_b(u)+\cI_c(u),
\end{gather*}
where $d_\Delta (s)=d(s,X_s^{1,N})-d(s,X_{s}^{2,N})$, $d\in\{a,b,c\}$, we have $\abs{d_\Delta (s)}\le C\aabs{X_s^{1,N}-X_s^{2,N}}$.
Thus,
\begin{gather*}
\Delta^2_t\le 3(\norm{\cI_a}^2_{2,t}+\norm{\cI_b}^2_{2,t}+\norm{\cI_c}^2_{2,t}.
\end{gather*}
The estimates will be similar to those of Theorem~\ref{thm-fund}, so we skip some minor details.

Estimate
\begin{gather*}
\norm{\cI_a}^2_{2,t}\le \int_0^t\bigg(\int_0^s \aabs{a_\Delta(u)}du\, ds + \int_0^s \int_{u}^{s} \abs{a_\Delta(v)} dv\, h(s,u)du\bigg)^2g(t,s) ds\\
\le C\int_0^t \bigg[\Delta_s + \bigg( \int_0^s \aabs{X^{1,N}_v-X^{1,N}_v} (s-v)^{-\alpha} dv\bigg)^2\bigg] g(t,s)ds\le C\int_0^t\Delta_s\,g(t,s)ds.
\end{gather*}
Further, $\norm{\cI_c(t)}^2_{2,t} \le  \int_0^t\left[\cI_c(s)^2 + \left(\int_0^s \abs{\cI_c(s)-I_c(u)}h(s,u)du\right)^2\right]g(t,s)ds=:I_c'+I_c''$.
Using \eqref{ineq} and \eqref{cocenka}, write
\begin{gather*}
I_c'\le C N \int_0^{t} \bigg[\int_0^s\bigg( \abs{c_\Delta(u)} u^{-\alpha} + \int_0^u \abs{c_\Delta(u)-c_\Delta(z)} h(u,z)dz\bigg)du\bigg]^2 g(t,s)ds
\\\le C_N\int_0^t \Delta_s\, g(t,s) ds
 + C_N\int_0^{t} \int_0^s\bigg(\int_0^u \Big(\aabs{X^{1,N}_u-X^{1,N}_z - X^{2,N}_u+X^{2,N}_z}\\ + (u-z)^\beta\abs{X^{1,N}_u- X^{2,N}_u}
 + \aabs{X^{1,N}_u-  X^{2,N}_u} \big(\aabs{X^{1,N}_{u,z}}+\aabs{X^{2,N}_{u,z}}\big) \Big)h(s,z) dz\bigg)^2du\,g(t,s)\,ds\\
\le C_N \int_0^t \Delta_s\, g(t,s) ds.
\end{gather*}
Similarly, by \eqref{ineq}  and \eqref{cocenka}
\begin{gather*}
I_c''\le C N \int_0^{t} \bigg[\int_0^{s}\int_0^s\bigg(\abs{c_\Delta(v)} (v-u)^{-\alpha} + \int_u^v \abs{c_\Delta(v)-c_\Delta(z)} h(v,z)dz\bigg)dv\,h(s,u)du\bigg]^2g(t,s)ds\\
\le C_N \int_0^{t} \bigg[\int_0^s\bigg(\aabs{X^{1,N}_v-X^{2,N}_v} (s-v)^{-2\alpha} + \int_0^v \abs{c_\Delta(v)-c_\Delta(z)} h(v,z)dz\bigg)(s-v)^{-\alpha}dv\bigg]^2g(t,s)ds\\
\le C_N\int_0^{t} \Delta_s\, g(t,s) ds +
 C_N\int_0^{t}\int_0^s\bigg( \int_0^v\Big(\aabs{X^{1,N}_{v,z} - X^{2,N}_{v,z}} + \abs{X^{1,N}_v- X^{2,N}_v} (v-z)^\beta
\\{} + \aabs{X^{1,N}_v-  X^{2,N}_v} \big(\abs{X^{1,N}_{v,z}}+\abs{X^{2,N}_{v,z}}\big) \Big)h(v,z) dz\bigg)^2 (s-v)^{-\alpha} dv\,g(t,s)ds
\le C_N \int_0^{t}\Delta_s g(t,s)\, ds.
\end{gather*}
Summing the estimates for $\cI_a(t)$ and $\cI_c(t)$ and using the Cauchy--Schwartz inequality, we get
\begin{gather*}
E\left[\big(\norm{\cI_a(t)}_{2,t} + \norm{\cI_c(t)}_{2,t}\big)^2\right]
\le C_N \int_0^t E\big[\Delta_s\big] g(t,s) ds.
\end{gather*}
Finally,
\begin{gather*}
E\big[\norm{\cI_b(t)}^2_{2,t}\big] \le 2\int_0^t E\left[\abs{\cI_b(s)}^2\right]g(t,s)ds + 2\int_0^t E\left[\bigg(\int_0^{s(N)}\!\abs{
\int_{u}^{s(N)}\!\! b_v\, dW_v} h(s,u)du\bigg)^2\right]g(t,s)ds\\
\le C \int_0^t \int_0^s E\left[b_{\Delta}(u)^2\right]du\, g(t,s)ds + C\int_0^t \int_0^s \int_u^s E\left[b_\Delta(v)^2\right] dv(s-u)^{-3/2-\alpha}du\,
g(t,s) ds  \\
\le C \int_0^t E[\Delta_s]g(t,s)ds.
\end{gather*}
As a result,
\begin{gather*}
E[\Delta_t] \le C_N \int_0^t E[\Delta_s] g(t,s) ds,
\end{gather*}
so by the generalized Gronwall lemma $E[\Delta_t]=0$ for all $t$. This implies
$X_t^1=X_t^2$ a.s. for all $t<\tau_N$. By letting $N\to\infty$ we get $X^1_t=X^2_t$ a.s. for all $t$.
\end{proof}

\end{document}